\documentclass{amsart}
\usepackage{latexsym}
\usepackage{amsmath}
\usepackage{amssymb}
\usepackage[all]{xy}

\newtheorem{thm}{Theorem}[section]
\newtheorem{prop}[thm]{Proposition}
\newtheorem{cor}[thm]{Corollary}
\newtheorem{lem}[thm]{Lemma}
\theoremstyle{definition}
\newtheorem{defn}[thm]{Definition}

\theoremstyle{remark}
\newtheorem{rem}[thm]{Remark}
\newtheorem{ex}[thm]{Example}

\newcommand{\h}[1]{%
H^*({#1}; {\mathbb K})}
\newcommand{\K}{{\mathbb K}}
\newcommand{\Q}{{\mathbb Q}}
\newcommand{\e}{\varepsilon}
\newcommand{\D}{\text{D}}
\newcommand{\mapright}[1]{%
\smash{\mathop{%
\hbox to 1cm{\rightarrowfill}}\limits_{#1}}}
\newcommand{\maprightd}[2]{%
\smash{\mathop{%
\hbox to 1.2cm{\rightarrowfill}}\limits^{#1}\limits_{#2}}}
\newcommand{\mapleft}[1]{%
\smash{\mathop{%
\hbox to 1cm{\leftarrowfill}}\limits_{#1}}}
\newcommand{\mapleftu}[1]{%
\smash{\mathop{%
\hbox to 0.8cm{\leftarrowfill}}\limits^{#1}}}
\newcommand{\maprightu}[1]{%
\smash{\mathop{%
\hbox to 1cm{\rightarrowfill}}\limits^{#1}}}
\newcommand{\maprightud}[2]{%
\smash{\mathop{%
\hbox to 1cm{\rightarrowfill}}\limits^{#1}_{#2}}}
\newcommand{\mapleftud}[2]{%
\smash{\mathop{%
\hbox to 1cm{\leftarrowfill}}\limits^{#1}_{#2}}}

\newcounter{eqn}[section]

\def\theeqn{\textnormal{(\thesection.\arabic{eqn})}}

\def\eqnlabel#1{%
\refstepcounter{eqn}%
\label{#1}%
\leqno{\theeqn}}

\begin{document}

\title[The (co)chain type levels]
{Upper and lower bounds of the (co)chain type level of a space} 

\footnote[0]{{\it 2000 Mathematics Subject Classification}: 
16E45, 18E30, 55R20, 13D07.
\\ 
{\it Key words and phrases.} Level, semi-free module,  
triangulated category, formal space, Borel construction, L.-S. category.


Department of Mathematical Sciences, 
Faculty of Science,  
Shinshu University,   
Matsumoto, Nagano 390-8621, Japan   
e-mail:{\tt kuri@math.shinshu-u.ac.jp}
}

\author{Katsuhiko KURIBAYASHI}
\date{}

\begin{abstract}
We establish an upper bound for the cochain type level of the total space of 
a pull-back fibration. It explains to us why the numerical invariants
for principal bundles over the sphere are less than or equal to
two. Moreover computational examples of the levels of path spaces and
Borel constructions, including biquotient spaces and
Davis-Januszkiewicz spaces, are presented.  We also show that the chain
type level of the homotopy fibre  of a map is greater than 
the E-category in the sense of Kahl, which is an algebraic
approximation of the Lusternik-Schnirelmann category of the map. 
The inequality fits between the grade and 
the projective dimension of the cohomology of the homotopy fibre.
\end{abstract}

\maketitle

\section{Introduction}
The {\it  level} of an object in a triangulated category  
was defined by Avramov, Buchweitz, Iyengar and Miller in \cite{ABIM}.  
The numerical invariant measures the number of steps to build the given
object from some fixed object via triangles. 
As for the level defined in the derived category  
$\D(A)$ of differential graded modules 
(DG modules) over a differential graded algebra (DGA) $A$, which is
viewed as a triangulated category \cite{Keller}, its important and
fundamental properties are investigated in \cite[Sections 3, 4 and 5]{ABIM}.
Moreover, these authors have established many lower bounds of the Loewy
length of a module over a ring $R$ by means
of the invariant level; see \cite[Introduction]{ABIM}. 

The level filters the smallest thick subcategory 
of a triangulated category ${\mathcal T}$ 
containing a given subcategory and hence the
invariant is regarded as a refinement of the notion of 
{\it finite building} for an object in ${\mathcal T}$ due to 
Dwyer, Greenlees and Iyengar \cite{D-G-I}; see also \cite{B-G} and  
\cite{G-H-S}.   
We also mention that the levels   
are closely related to the notion of
dimensions of triangulated categories; see \cite[2.2.4]{ABIM}, 
\cite{B-B} and \cite{R}. 

To study topological spaces with categorical representation theory, 
we were looking for an appropriate invariant which stratifies the
category of topological spaces in some sense and found the
invariant level at last.  Thus a topological invariant of a space $X$ 
over a space $B$, which is called the {\it cochain type level of} $X$ over the space $B$,  
was introduced in \cite{K2}. 

Let $C^*(B; \K)$ be the normalized 
singular cochain algebra of a space $B$ 
with coefficients in a field $\K$. Then the level of $X$ over a space
$B$ is defined to be the level in the sense of \cite{ABIM} of the DG module 
$C^*(X; \K)$ over the DG algebra $C^*(B; \K)$ in the triangulated category 
 $\D(C^*(B; \K))$; see Section 2 for more details.      
It turns out that the level of $X$ characterizes indecomposable
elements of $\D(C^*(B; \K))$ which make up the $C^*(B; \K)$-module 
$C^*(X; \K)$ in the triangulated category. Such constitutions are called 
{\it molecules} of $C^*(X; \K)$ in \cite{K2}. 
In order to make the observation more clear, we  recall 
some properties of the triangulated category $\D(C^*(B; \K))$. 

By applying Auslander-Reiten theory for derived categories \cite{H} \cite{H2}, 
J{\o}rgensen \cite{J} \cite{J2} has clarified the structure of the 
Auslander-Reiten quiver of the full
subcategory 
$\D^c(C^*(B; \K))$ of compact objects of $\D(C^*(B; \K))$ 
provided the space $B$ is 
Gorenstein at $\K$ in the sense of F\'elix, Halperin and Thomas
\cite{FHT_GSpace}.   
In fact, the result \cite[Theorem 0.1]{J2} tells us that each component
of the quiver is of the form ${\mathbb Z}A_\infty$; see also \cite{J}
and \cite{Schmidt}. 
Depending on the detailed information of the quiver of 
$\D^c(C^*(S^d; \K))$, 
the computation of the level of an appropriate 
space over the sphere $S^d$ is performed in \cite{K2}.  
In particular, we see that the level of a space $X$ over $S^d$ is less
than or equal to an integer $l$ if and only if the DG module $C^*(X; \K)$ over
$C^*(S^d; \K)$ is made up of molecules lying between the $l$th horizontal
line and the bottom one of the quiver; 
see  \cite[Proposition 6.6]{Schmidt},  \cite[Examples 5.2 and 5.3]{K2}  
and \cite[Theorem 8.13]{J}. 

On the other hand, the result \cite[Theorem 2.12]{K2} asserts that 
there exists just one vertex in 
the Auslander-Reiten quiver which is realized by a space over $S^d$ 
via the singular cochain functor. This means that   
if the level of a space $X$ over $S^d$ is greater than or equal to three, 
then the DG module $C^*(X; \K)$ consists of at least two
molecules in $\D^c(C^*(S^d; \K))$; see \cite[Theorem 2.6]{K2}.  
Moreover the result \cite[Proposition 2.4]{K2} implies that all of 
the levels of total spaces of 
principal $G$-bundles over the $4$-dimensional sphere   
are less than or equal to two if the cohomology 
of the classifying space of $G$ is isomorphic to
a polynomial algebra on generators with even degree. 

As mentioned above, the level of a DG module $M$ in the triangulated
category $\D(A)$ of a DG algebra $A$ counts the number
of steps to build $M$ out of, for example, $A$ via triangles in $\D(A)$. 
In \cite[Proposition 2.6]{K2}, it is shown that the cochain type level
gives a lower estimate of the number of a pile of 
rational spherical fibrations. 
Thus an important issue is to clarify further topological quantity which
the level measures. 

As a first step, many computations of levels might be needed.  
In this paper, we present a method for computing the levels of 
spaces. In particular, we obtain an upper bound for the level of the
corner space of a fibre square; see Theorem \ref{thm:main_new}. 
Moreover,  we try to compute the level of path spaces and
Borel constructions, including biquotient spaces \cite{Sin} and
Davis-Januszkiewicz spaces \cite{DJ} \cite{Pa}. 

We also introduce the {\it chain type level} of a space and consider 
the relationship between the level and other topological invariants.  
Especially, the chain type level of the homotopy fibre of a map $f$
gives an upper estimate for the E-category in the sense of Kahl \cite{Kahl}, 
which is an algebraic approximation of 
the Lusternik-Schnirelmann category (L.-S.category) of $f$; 
see Theorem \ref{thm:cat-level}. This is one of the remarkable results on
the level. 
Thus we can bring the notion of the level into the study of L.-S. categories and their related invariants. 
It turns out that the L.-S. category of a simply-connected
rational space $X$ has an upper bound described 
in terms of the chain type level associated with the space $X$; 
see Corollary \ref{cor:R-case}. 
This is an answer to a topological description of the level.

\section{The (co)chain type levels and main theorems}

In this section, our results are stated in detail. 
We begin by recalling the explicit definition of the level.  

Let $\D$ be a triangulated category and ${\mathcal A}$ a subcategory of
$\D$.  We denote by $\text{\tt add}^{\Sigma}({\mathcal A})$ the smallest 
strict full subcategory of $\D$ that contains ${\mathcal A}$ and is closed under finite 
direct sums and all shifts. The category  $\text{\tt smd}({\mathcal A})$
is defined to be the smallest full subcategory of $\D$ that contains ${\mathcal A}$
and is closed under retracts.   
For full subcategories ${\mathcal A}$ and ${\mathcal B}$ of $\D$, let 
${\mathcal A}*{\mathcal B}$ be 
the full subcategory whose objects $L$ occur in a triangle 
$
M \to L \to N \to \Sigma M
$ 
with $M \in {\mathcal A}$ and $N \in {\mathcal B}$.  
We define $n$th thickening  $\text{\tt thick}^n_{\D}({\mathcal C})$ of a 
full subcategory ${\mathcal C}$ by 
$$
\text{\tt thick}^n_{\D}({\mathcal C}) = 
\text{\tt smd}((\text{\tt add}^{\Sigma}({\mathcal C}))^{*n}),  
$$ 
where $\text{\tt thick}^0_{\D}({\mathcal C}) = \{0\}$; see 
\cite{B-B} and \cite[2.2.1]{ABIM}.

Let $A$ be a DG algebra over a field and $\D(A)$ the triangulated
category of DG modules over $A$ \cite{Keller}. 
We then define a numerical invariant $\text{level}_{\D(A)}(M)$ for an
object $M$ in $\D(A)$, which is
called the {\it level of} $M$,  by 
$$
\text{level}_{\D(A)}(M):= \inf \{n \in {\mathbb N}  \ |
M \in  \text{\tt thick}^{n}_{\D(A)}(A) \}. 
$$
If no such integer exists, we set $\text{level}_{\D(A)}(M)=\infty$. 
Here $A$ is regarded as the full subcategory of $\D(A)$ consisting
of the only object $A$.  We refer the reader to \cite[2.1]{ABIM} for
the levels defined in  more general triangulated categories and their
fundamental properties.

In what follows, let $\K$ be a field of arbitrary characteristic and 
all coefficients of (co)chain complexes are in $\K$. 
Moreover, unless otherwise specified, 
it is assumed that a space has the homotopy
type of a CW complex whose cohomology with coefficients in the
underlying field is locally finite; that is, the $i$th cohomology is of
finite dimension for any $i$.

Let $B$ be a space and $\mathcal{TOP}_B$ the category of maps with the
target $B$. For any object $f : X \to B$,    
the normalized singular cochain $C^*(X; \K)$ of the source space $X$ of $f$ is regarded as
a DG module over the cochain algebra $C^*(B; \K)$ 
via the induced map $C^*(f) : C^*(B; \K) \to C^*(X; \K)$. 
Thus the cochain gives rise to a contravariant functor from the category 
$\mathcal{TOP}_B$ to the triangulated category $\D(C^*(B; \K))$:    
$$
C^*(s(-) ; \K) :  \mathcal{TOP}_B \to \D(C^*(B; \K)), 
$$
where $s(f)$ for $f$ in $\mathcal{TOP}_B$ denotes the source of
$f$. 
We say that a morphism $\varphi : f \to g$ in $\mathcal{TOP}_B$ is 
a weak equivalence if so is the underlying map  $\varphi : s(f) \to s(g)$. 
We write $\text{level}_{\D(C^*(B; \K))}(s(f))$ for $\text{level}_{\D(C^*(B; \K))}(C^*(s(f); \K))$ and
refer to it as the {\it cochain type level} of the space $s(f)$. Since a
weak equivalence in $\mathcal{TOP}_B$ induces a quasi-isomorphism of 
$C^*(B; \K)$-modules, it follows that the cochain type level is a
numerical homotopy invariant.  

Let $F_f$ be the homotopy fibre of a map $f : X \to B$. 
The Moore loop space $\Omega B$ acts on the space $F_f$ 
by the holonomy action. Thus the normalized chain complex $C_*(F_f; \K)$ is a DG
module over the chain algebra $C_*(\Omega B; \K)$. 
The chain and the homotopy fibre construction enable us 
to obtain a covariant functor  
$$
C_*(F_{(-)} ; \K) :  \mathcal{TOP}_B \to \D(C_*(\Omega B; \K))
$$
from the category $\mathcal{TOP}_B$ 
to the triangulated category $\D(C_*(\Omega B; \K))$.     
We then define the {\it chain type level} of the space $F_f$ by 
$\text{level}_{\D(C_*(\Omega B; \K))}(C_*(F_f; \K))$ and denote it 
by $\text{level}_{\D(C_*(\Omega B; \K))}(F_f)$.   
It is immediate that the chain type level is also a  
numerical topological invariant for objects  in 
$\mathcal{TOP}_B$ with respect to weak equivalences.

We first examine especially the cochain type levels 
of spaces $E_{\varphi}$ 
which fit into any of the fibre squares ${\mathcal F}_1$, ${\mathcal F}_2$ and 
${\mathcal F}_3$ explained below.   
Let $B$ be a space with basepoint $*$ and $B^I$ the space of all maps 
from the interval $[0, 1]$ to $B$ with the compact-open topology.  
Let $PB$ denote the subspace of $B^I$ of all paths ending at $*$.
We define a map $\e_i : B^I \to B$ by $\e_i(\gamma)=\gamma(i)$ for 
$i=0$ and $1$. Then one obtains 
fibre squares ${\mathcal F}_1$ and ${\mathcal F}_2$ of the forms 
$$
\xymatrix@C20pt@R8pt{
E_{\varphi} \ar[r] \ar[dd] & PB \ar[dd]^{\e_0}  &  &
              E_{\varphi} \ar[r] \ar[dd] & B^I \ar[dd]^{\e_0\times \e_1} \\
          & & \text{and} &   \\ 
X  \ar[r]^{\varphi}  & B  &  &  X  \ar[r]^(0.4){\varphi}  & B\times B,   
}
$$
respectively. Observe that $E_\varphi$ in ${\mathcal F}_1$ is nothing but
the homotopy fibre of the map $\varphi : X \to B$. 
In particular if the map $\varphi$ in ${\mathcal F}_2$ is the diagonal map 
$B \to B\times B$, then $E_\varphi$ is the free loop space; see
\cite{S} and \cite{K-Y} for applications of the fibre square 
to the computation of the cohomology of a free loop space. 

Let $G$ be a connected Lie group and $H$ a closed subgroup of 
$G\times G$. Let $\delta G$ denote the closed subgroup defined by 
$\delta G = \{(g,g)\in G\times G  \ | \ g\in G \}$. 
Then one has a fibre square ${\mathcal F}_3$ of the form
$$
\xymatrix@C25pt@R25pt{
E_{\varphi} \ar[r] \ar[d] & E(G\times G)/\delta G \ar[d]^{q} \\
BH  \ar[r]^(0.4){\varphi}  & B(G\times G),    
}
$$
where $\varphi$ denotes the map induced by the inclusion 
$j : H \to G\times G$ between 
the classifying spaces; see \cite[Section 4]{JHE}. Here the total space 
$E_\varphi$ is the Borel construction $E(G\times G)\times_{H}G$ 
associated with the action $H\times G \to G$ 
defined by $(h, k)g = hgk^{-1}$ for $(h, k) \in H$ and $g \in G$. 
We mention that this total space is homotopy equivalent to 
a double coset manifold under some hypotheses; 
see \cite{JHE} and \cite[(1.7), (2.2) Proposition]{Sin} for more details.  

In the fibre squares ${\mathcal F}_1$ and ${\mathcal F}_2$, 
if the space $B$ is simply-connected and satisfies the condition that 
$\dim H^*(B; \K) < \infty$, then the cohomology $H^*(\Omega B; \K)$ 
of the fibre is of infinite dimension. This follows from the Leray-Serre
spectral sequence argument for the path-loop fibration 
$\Omega B \to PB \to B$. 
Therefore the results 
\cite[Lemma 3.9, 6.3.2]{Schmidt} allow us to conclude that 
$\text{level}_{\D(C^*(X; \K))}(E_\varphi) =\infty$. Then in
this paper we shall confine ourselves to considering the cochain type level of the
space $E_\varphi$ in the case where $H^*(B; \K)$ is a polynomial
algebra.  

The first result is concerned with an upper bound of the cochain type
level of the corner space $E_\varphi$ in any of fibre squares
${\mathcal F}_1$,  ${\mathcal F}_2$ and ${\mathcal F}_3$. 
To describe the result precisely, 
we recall from \cite{K1} an important  class of pairs of maps. 
We say that a space $X$ is {\it $\K$-formal} 
if it is simply-connected and there exists a
quasi-isomorphism to the cohomology $H^*(X; \K)$ 
from a minimal $TV$-model for $X$ in the sense of Halperin 
and Lemaire \cite{H-L}; see also \cite{E}. 
In this case we have a sequence of quasi-isomorphisms
$$
\xymatrix@C25pt@R20pt{
\h{X} & TV_X  \ar[l]_(0.4){\phi_X}^(0.4){\simeq} 
\ar[r]^(0.4){m_X}_(0.4){\simeq}  & C^*(X; {\mathbb K}),
} 
$$
where $m_X : TV_X \to C^*(X; {\mathbb K})$ denotes a minimal 
$TV$-model for $X$. 
Let $q : E \to B$ and $\varphi : X \to B$ be maps between 
$\K$-formal spaces. Then the pair $(q, \varphi)$ is called 
{\it relatively $\K$-formalizable} 
if there exists a commutative diagram up to homotopy 
$$
\xymatrix@C25pt@R15pt{
\h{E}  & TV_E \ar[l]_(0.4){\phi_E}^(0.4){\simeq} 
\ar[r]^(0.4){m_E}_(0.4){\simeq} & C^*(E; {\mathbb K})  \\
\h{B} \ar[u]^{H^*(q)} \ar[d]_{H^*(\varphi)}
& TV_B  \ar[l]_(0.4){\phi_B}^(0.4){\simeq} \ar[r]^(0.4){m_B}_(0.4){\simeq} 
\ar[u]_{\widetilde{q}} \ar[d]^{\widetilde{\varphi}}
& C^*(B; {\mathbb K}) \ar[u]_{q^*} \ar[d]^{\varphi^*} \\ 
\h{X} & TV_X  \ar[l]_(0.4){\phi_X}^(0.4){\simeq} 
\ar[r]^(0.4){m_X}_(0.4){\simeq}  & C^*(X; {\mathbb K}) ,
}
$$
in which horizontal arrows are quasi-isomorphisms. 
We call a map $q : E \to B$  $\K$-{\it formalizable} if $(q, \iota)$ is 
a relatively $\K$-formalizable pair 
for some constant map $\iota : * \to B$.

For a graded algebra $A$, let $A^+$denote the ideal 
$\oplus_{i\geq 1}A^i$. We write $QA$ for the vector space of
indecomposable elements, namely $QA=A/(A^+\cdot A^+)$. Observe that 
the vector space $QA$ is viewed as a subspace of $A$.

It follows from the proof of \cite[Theorem 1.1]{K1} that 
a pair $(q, \varphi)$ of maps between $\K$-formal spaces 
with the same target is relatively $\K$-formalizable if the two maps 
satisfy any of the following three conditions $(P_1)$, $(P_2)$ and $(P_3)$ 
concerning a map $\pi : S \to T$ respectively. 

\medskip
\noindent
$(P_1)$ $\h{S}$ and $\h{T}$ are polynomial algebras with at most countably
many generators in which the operation 
$Sq_1$ vanishes when the characteristic of the field $\K$ is 2. 
Here $Sq_1x = Sq^{n-1}x$ for $x$ with degree $n$; see \cite[4.9]{Mu}. 
\noindent

\noindent
$(P_2)$ The homomorphism $B\h{\pi} : B\h{T} \to B\h{S}$ defined by 
$\h{\pi}$ between the bar complexes induces 
an injective homomorphism on the homology. 

\noindent
$(P_3)$ $\widetilde{H}^i(S; {\mathbb K})=0$ 
for any $i$ with 
$\dim\widetilde{H}^{i-1}(\Omega T; {\mathbb K}) 
- \dim (QH^*(T;  {\mathbb K}))^i \neq 0$, where $\widetilde{H}^*(X; {\mathbb K})$ 
denotes the reduced cohomology of a space $X$.   

\medskip

The following examples show that some important maps enjoy
$\K$-formalizability. 

\begin{ex}
\label{ex:examples} 
(i) Let $G$ be a connected Lie group and $K$  a connected subgroup. 
Suppose that $H_*(G; {\mathbb Z})$ and $H_*(K; {\mathbb Z})$ are
$p$-torsion free. Then the map $Bi : BK \to BG$ between classifying
spaces induced by the
inclusion  $i : K \to G$ satisfies the condition $(P_1)$ with respect
to the field ${\mathbb F}_p$. Assume further that 
$\text{rank} \ G = \text{rank} \ K$. Let $M$ be the homogeneous space 
$G/K$ and $\text{aut}_1(M)$ the connected component of 
function space of all self-maps on $M$ containing the identity map.  
Then the universal fibration 
$\pi : M_{\text{aut}_1(M)} \to  B_{\text{aut}_1(M)}$
with fibre $M$ satisfies the condition  $(P_1)$ with respect
to the field ${\mathbb Q}$; see \cite{Hau} and \cite{K3}. \\
(ii) Let $q : E \to B$ be a map between $\K$-formal spaces with a 
section. Then $q$ satisfies the condition  $(P_2)$. This follows from
the naturality of the bar construction. \\
(iii) Consider a map $f : S^4 \to BG$ for which $G$ is a
simply-connected Lie group and $H_*(G; {\mathbb Z})$ is $p$-torsion free.
Suppose that 
$\widetilde{H}^i(S^4; {\mathbb F}_p)\neq0$, then $i=4$. 
One obtains  $\dim\widetilde{H}^{4-1}(\Omega BG; {\mathbb F}_p) 
- \dim (QH^*(BG; {\mathbb F}_p))^4 = 0$.  Thus 
the map $f : S^4 \to BG$ satisfies the condition $(P_3)$. 
\end{ex}

One of our results is described as follows. 

\begin{thm}\label{thm:main_new}
Let ${\mathcal F}$ be a pull-back diagram
$$
\xymatrix@C25pt@R20pt{
E_{\varphi} \ar[r] \ar[d] & E \ar[d]^{q} \\
X  \ar[r]^{\varphi}  & B   
}
$$
in which $q$ is a fibration and 
the pair $(q, \varphi)$ is relatively $\K$-formalizable.   
Suppose that either of the following conditions 
{\em (i)} and {\em (ii)} holds. 

{\em (i)} The cohomology $H^*(B; \K)$ is a polynomial 
algebra generated by $m$ indecomposable elements.  
Let $\Lambda$ be the subalgebra of $H^*(B; \K)$ generated by 
the vector subspace $\Gamma:=\text{\em Ker} \ \varphi^* \cap QH^*(B; \K)$. 
Then $\dim \text{\em Tor}^{\Lambda}_*(H^*(E; \K), \K) < \infty$. 

{\em (ii)} There exists a homotopy commutative diagram 
$$
\xymatrix@C25pt@R15pt{
E \ar[d]_q  \ar[r]^{\simeq} & B'\ar[d]^{\Delta}  \\
B  \ar[r]^(0.4){\simeq}_(0.4)h   & B' \times B'     
}
$$
in which horizontal arrows are homotopy equivalences and $\Delta$ is
the diagonal map. Moreover $H^*(B'; \K)$ is a polynomial 
algebra generated by $m$ indecomposable elements.  
In this case put 
$\Gamma = \text{\em Ker} \ (\Delta^* |_{QH^*(B'\times B')})
\cap \text{\em Ker}\ (h\varphi)^*$. 

\noindent
Then one has 
$$
\text{\em level}_{\text{\em D}(C^*(X; \K))}(E_{\varphi} )
\leq m-\dim \Gamma +1.
$$   
In particular, 
$\text{\em level}_{\text{\em D}(C^*(X; \K))}
(E_{\varphi})=1$ if $\varphi^*\equiv 0$.  
\end{thm}


We are able to characterize a space of level one with a spectral
sequence. 
\begin{prop}
\label{prop:main2}
Let ${\mathcal F}' : F \stackrel{j}{\to} E \to B$ be a fibration with B
simply-connected and $F$ connected. If  
$\text{\em level}_{\text{\em D}(C^*(B; \K))}(E)=1$,
then both the Leray-Serre spectral sequence and 
the Eilenberg-Moore spectral sequence for  ${\mathcal F}'$ collapse at
the $E_2$-term, where the coefficients of the spectral sequence are in 
the field $\K$. 
\end{prop}

\begin{rem}
\label{rem:level} 
Let $G$ be a simply-connected Lie group.  As mentioned above, 
with the aid of Auslander-Reiten theory over spaces by 
J{\o}rgensen \cite{J}\cite{J2}\cite{J3}, we have determined  
the level $L:=\text{level}_{\D(C^*(S^4; \K))}(E_\varphi)$ 
for the total space of the $G$-bundle over $S^4$ with the
classifying map $\varphi : S^4 \to BG$ provided $H^*(BG; \K)$ is 
a polynomial algebra on generators with even degree. 
The result \cite[Proposition 2.4]{K2} asserts that $L=2$ if
$\varphi^*\neq 0$ and $L=1$ otherwise. 
Though the computations in \cite{K2} are ad hoc, the result is not
 accidental since it is deduced from 
Theorem \ref{thm:main_new} and Proposition \ref{prop:main2}. 

In fact, let $p : EG \to BG$ be the universal bundle.  
The maps $\varphi$ and $p$ satisfy the condition $(P_3)$, 
respectively so that the pair 
$(\varphi, p)$ is relatively $\K$-formalizable; 
see Example \ref{ex:examples}.
Since $EG$ is contractible, the condition (i) in Theorem \ref{thm:main_new}
holds. Thanks to the theorem, we have $L\leq 2$ if $\varphi^*\neq 0$ 
and $L=1$ otherwise because $\dim \Gamma= \dim QH^*(BG)-1$ 
if $\varphi^*\neq 0$. 

Suppose that $\varphi^*\neq 0$. In the Leray-Serre spectral sequence 
$\{E_r^{*,*}, d_r\}$ for the universal bundle $G \to EG \to BG$, 
the indecomposable elements of 
$H^*(G; \K)\cong E_2^{0, *}$ are chosen as transgressive ones. 
Since  $\varphi^*\neq 0$, it follows that the Leray-Serre 
spectral sequence for the fibration $G \to E_\varphi \to S^4$ does not
collapse at the $E_2$-term. Proposition \ref{prop:main2} implies that
$L\neq 1$. We have $L=2$.  
\end{rem}

Let us mention that the original proof of \cite[Proposition 2.4]{K2}
enables us to obtain the indecomposable objects in $\D(C^*(S^4))$
which construct the DG module $C^*(E_\varphi)$ over $C^*(S^4)$. 
As mentioned in the introduction, 
such objects are called {\it molecules} because they are viewed as 
structural ones smaller than cellular cochains; see 
\cite[Section 2, Example 6.3]{K2}.

In general, taking shifts and direct sums of objects with the same level
leave the invariant unchanged. From this fact one deduces 
the following noteworthy result which states that 
the cochain type level of a Borel construction 
associated with Lie groups coincides 
with that of the construction with their maximal tori.

\begin{thm}
\label{thm:Borel}
Let $G$ be a connected Lie group, $T_H$ and
$T_K$ maximal tori of subgroups $H$ and $K$ of $G$, respectively.  
Suppose that $H^*(BG; \K)$, $H^*(BH; \K)$ and $H^*(BK; \K)$ are 
polynomial algebras with generators of even dimensions. 
Then 
$$
\text{\em level}_{\text{\em D}(C^*(BH; \K))}(EG\times_HG/K)
=\text{\em level}_{\text{\em D}(C^*(BT_H; \K))}
(EG\times_{T_H}G/T_K).
$$
\end{thm}

In the rest of this section, we focus on the chain type levels of
spaces. 

Let $\mathcal{DGM}$ be the category of supplemented differential 
graded modules over $\K$; that is, an object $M$ is of the form 
$M=\K\oplus \overline{M}$ where $d\K = 0$ and 
$d(\overline{M}) \subset \overline{M}$.  
Let $A$ be a monoid object in $\mathcal{DGM}$, namely a differential
graded algebra. We denote by $A^\natural$ the underlying graded algebra
of $A$. 


In \cite{Kahl}, Kahl introduced three notions of algebraic 
approximations of the L.S.-category of a map as numerical 
invariants in monoidal cofibration categories; see also \cite{Kahl2}.   
We here confine ourselves to treating such notions in 
$\mathcal{DGM}\text{-}A$,  the category of supplemented
differential graded right $A$-modules. 
Then the chain type level of a space is related to the E-category, which is one of the approximations. 
In order to describe the result, 
we first recall the definition of the E-category of an object in 
$\mathcal{DGM}\text{-}A$.  

Let $B(\K, A, A) \to \K  \to 0$ be the bar resolution of $\K$ as a right 
$A$-module. Observe that $B(\K, A, A) = T(\Sigma\overline{A})\otimes A$ 
as a $A^\natural$-module, where $\overline A$ 
is the augmentation ideal of $A$, 
$(\Sigma\overline{A})_n = \overline{A}_{n-1}$ and 
$T(W)$ denotes the tensor coalgebra generated by a
vector space  $W$. Define a sub $A$-module $E_nA$ of $B(\K, A, A)$  
by $E_nA= T(\Sigma\overline{A})^{\leq n}\otimes A$. 

\begin{defn} \cite{Kahl}
The {\it E-category} for $M$ in $\mathcal{DGM}\text{-}A$, denoted 
$\text{Ecat}_AM$, is the least integer $n$ for which there exists a
morphism 
$M \to E_nA$ in the homotopy category of $\mathcal{DGM}\text{-}A$. 
If there is no such integer, then we set  $\text{Ecat}_AM=\infty$. 
\end{defn}

Let $R$ be a graded algebra over $\K$ and $M$ a graded module over
$R$. Then the {\it grade} of $M$, denoted $\text{grade}_RM$, is defined
to be the least integer $k$ such that $\text{Ext}_R^k(M, R)\neq 0$. 
If $\text{Ext}_R^*(M, R)= 0$, then we set 
$\text{grade}_RM=\infty$. The {\em projective dimension} of $M$, denoted 
$\text{pd}_RM$, is defined to be the least integer $k$ such that $M$
admits a projective resolution of the form 
$0 \to P_k \to P_{k-1} \to \cdots \to P_0 \to M \to 0$.  
We set $\text{pd}_RM=\infty$ if no such integer exists. By definition, it is immediate that 
$\text{grade}_RM \leq \text{pd}_RM$. 


The grade and the projective dimension are numerical invariants which
appear in homological algebra. The E-category is an
invariant described, in general, in terms of homotopical algebra; see
\cite[Definition 2.1]{Kahl}.  The level is a numerical invariant defined
in a triangulated category as is seen above. 
These invariants and the L.-S. category of a map meet with 
inequalities in the following theorem and the ensuing remark. 

\begin{thm}
\label{thm:cat-level} Let $f : X \to Y$ be a map from a connected space to a
simply-connected space. Then one has 
$$
\text{\em grade}_{H_*(\Omega Y)}H_*(F_f) \! \leq \! 
\text{\em Ecat}_{C_*(\Omega Y)}C_*(F_f) \! \leq \!
\text{\em level}_{\text{\em D}(C_*(\Omega Y))}(F_f) - 1 \! \leq \!
\dim H^*(X)-1.  
$$
Assume further that 
$\dim \text{\em Tor}_{-i}^{H_*(\Omega Y)}(H_*(F_f), \K) < \infty$ for any 
$i \geq 0$. Then 
$$
\text{\em level}_{\text{\em D}(C_*(\Omega Y))}(F_f) - 1 \leq 
\text{\em pd}_{H_*(\Omega Y)}H_*(F_f).   
$$
\end{thm}

\begin{rem} Let $f : X \to Y$ be a map from a connected space to a
simply-connected space. Then it follows from 
\cite[Theorems 2.7 and 3.5]{Kahl} that the E-category 
$\text{Ecat}_{C_*(\Omega Y)}C_*(F_f)$ is less than or equal to 
the L.-S. category of $f$. 
Indeed, the result \cite[Theorem 35.9]{F-H-T} 
due to F\'elix, Halperin and Thomas asserts  that 
$\text{grade}_{H_*(\Omega Y)}H_*(F_f) \leq \text{cat} f$ without assuming that 
$Y$ is simply-connected. Moreover the latter half of the result implies
that, if $\text{cat}f = \text{grade}_{H_*(\Omega Y)}H_*(F_f)$, then 
the value coincides with $\text{pd}_{H_*(\Omega Y)}H_*(F_f)$. This
yields that 
$$
\text{cat}f = \text{grade}_{H_*(\Omega Y)}H_*(F_f)=
\text{level}_{\text{D}(C_*(\Omega Y))}(F_f) - 1=
\text{pd}_{H_*(\Omega Y)}H_*(F_f)
$$
provided $\text{cat}f = \text{grade}_{H_*(\Omega Y)}H_*(F_f)$ and 
$\text{Tor}_{-i}^{H_*(\Omega Y)}(H_*(F_f), \K)$ is of finite
dimension for any $i$. 
\end{rem}

The result \cite[Theorem 8.3]{Kahl} enables us to conclude 
that the E-category coincides with the 
M-category of a map $f : X \to Y$ between simply-connected spaces 
in the sense of Halperin and Lemaire \cite{H-L} and Idrissi \cite{Id}: 
$\text{Ecat}_{C_*(\Omega X)}C_*(F_f)=\text{Mcat}f$. 
Thus Theorem \ref{thm:cat-level} 
gives upper bounds of the M-category.

With the aid of the fascinating theorem due to Hess \cite{Hess},  
we moreover have a remarkable result on the L.-S. category of a rational space. 

\begin{cor} 
\label{cor:R-case}
Let $X$ be a simply-connected rational space. Then 
$$
\text{\em grade}_{H_*(\Omega X; \Q)}\Q \leq 
\text{\em cat} X 
\leq 
\text{\em level}_{\text{\em D}(C_*(\Omega X; \Q))}\Q - 1  
\leq \dim H^*(X; \Q)-1. $$
\end{cor}

Thanks to Theorem \ref{thm:cat-level} 
and Corollary \ref{cor:R-case}, computational examples of 
chain type levels can be  obtained; see Examples \ref{ex:5.1} and \ref{ex:5.2}.

An outline for the rest of the article is as follows. In the third section, 
after fixing notations and terminology for this article, 
we recall fundamental properties of the level of  DG-modules. 
In Section 4, we prove
Theorem \ref{thm:main_new} and Proposition \ref{prop:main2}.  
A corollary and a variant of Theorem \ref{thm:main_new} are also
established. 
In Section 5, by means of our results and general
theory for levels developed in \cite{ABIM}, 
we consider the numerical invariant for path spaces, 
biquotient spaces \cite{JHE}\cite{Sin}  
and Davis-Januszkiewicz spaces \cite{DJ}\cite{Pa} 
which appear in toric topology. Theorem \ref{thm:Borel} is proved in
this section. Section 6 is devoted to proving Theorem \ref{thm:cat-level} 
and Corollary \ref{cor:R-case}. 
We consider other lower and upper estimates for the level 
in Section 7. 

\section{Preliminaries}

Let $\K$ be a field of arbitrary characteristic. 
A {\it graded module} is a family $M =\{M^i\}_{i \in {\mathbb Z}}$ of $\K$-modules and a {\it differential} in $M$ is a linear map 
$d_M : M^i \to M^{i+1}$ of degree $+1$ such that $d^2=0$. We use the notation 
$M_i = M^{-i}$ to write $M=\{M_i\}_{i \in {\mathbb Z}}$. 
Following \cite[Appendix]{FHT_GSpace}, we moreover use convention and terminology in differential homological algebra; 
see also \cite[Section 1]{FHT} and \cite[Appendix]{H-L}.  

We here recall from \cite[Part III, 1]{K-M} the mapping cone construction. Let $I$ denote the unit interval $\K$-module; that is, 
it is free on generators $[0], [1] \in I^0$ and $[I] \in I^{-1}$ with $d([I])=[0]-[1]$. 
Let $A$ be a differential graded algebra (DG algebra) with the
underlying graded algebra 
$A^\natural$ and 
$X$ a differential graded right $A$-module (DG module). The {\it cone} $CX$ is defined to be the quotient module 
$(I/\K\{[1]\})\otimes X$. We define the {\it suspension} $\Sigma X$ by 
$\Sigma X = (I/\partial I)\otimes X$, where $\partial I$ denotes the DG submodule of $I$ generated by $[0]$ and $[1]$. 
Observe that $(\Sigma M)^n \cong M^{n+1}$.  We then have the mapping cone construction $(C_f, d)$ which is defined by 
$C_f = Y \oplus \Sigma X$ and 
$d= \left(
\begin{array}{cc}
d_Y & f \\
0 & -d_X
\end{array}
\right).
$
Observe that, by definition, triangles in $\D(A)$ come from the
sequences of the form 
$X \stackrel{f}{\to}  Y  \to C_f \to \Sigma X$
via the localization functor from the category of differential graded
right $A$-modules to the derived category $\D(A)$. 

Let $f : X \to Y$ be a morphism of DG $A$-modules. Consider the pushout diagram 
$$
\xymatrix@C25pt@R15pt{
X \ar[r]^f \ar[d]_i & Y \ar[d] \\
CX \ar[r] & Y\cup_fCX &
}
$$ in which $i : X \to CX$ denotes the natural inclusion and by definition 
$Y\cup_fCX= Y\oplus CX/((0, [0]\otimes x) - (f(x), 0) ; x \in X)$.
Define a map $\gamma : Y\oplus CX \to C_f$ 
by $\gamma(0, [I]\otimes x) =(0, x)$, $\gamma(y, 0)=(y, 0)$ 
and $\gamma(0, [0]\otimes x) =(f(x), 0)$ for $x \in X$ and $y \in Y$. 
It follows that $\gamma$ gives rise to an isomorphism 
$\overline{\gamma} : Y\cup_fCX \to C_f$ of differential graded right $A$-modules.  

Let $F$ be a DG-module over $A$ and $F'$ a sub DG-module of $F$ such that 
the quotient $F/F'$ is isomorphic to a coproduct of shifts of $A$, say 
$$
F/F' \cong \bigoplus_{i^\in J} \Sigma^{l_i}A \cong \Sigma (Z\otimes A), 
$$ 
where $Z$ denotes a graded vector space $\oplus_{i\in J} \Sigma^{l_i-1}\K$. 
Then $F$ is isomorphic to a right $A^\natural$-module of the form $F'\oplus \Sigma(Z\otimes A)$. 
It follows that $F$ fits in the pushout diagram 
$$
\xymatrix@C25pt@R15pt{
Z\otimes A \ar[r]^\xi \ar[d]_i & F' \ar[d] \\
C(Z\otimes A) \ar[r] & F'\cup_\xi C(Z\otimes A)\cong F &
}
$$ in which $\xi$ is a morphism of DG-modules over $A$ defined by 
$$
\xi(z\otimes a) = d(\Sigma(z\otimes a))-(-1)^{\deg \Sigma z}\Sigma z \otimes da
= d(\Sigma(z\otimes 1))a
$$
for $z\otimes a \in Z\otimes A$. 

We recall results concerning the level, which are used frequently in the rest of this paper. 
The first one is useful when considering the cochain type levels 
of spaces over a $\K$-formal space. 

\begin{thm}
\label{thm:level_formal} 
\cite[Theorem 1.3]{K2}
Let ${\mathcal F}$ be a fibre square as in Theorem \ref{thm:main_new} 
for which $(q, \varphi)$ is relatively $\K$-formalizable.  Then one has 
$$
\text{\em level}_{\text{\em D}(C^*(X; \K))}(E_\varphi) 
= \text{\em level}_{\text{\em D}(H^*(X; \K))}(H^*(E;\K)\otimes_{H^*(B;\K)}^{\mathbb L}H^*(X; \K)).  
$$
\end{thm}

The level of a DG module $M$ is evaluated with the length of a semi-free
filtration of $M$. 

\begin{defn} \cite[4.1]{ABIM}\cite{FHT_GSpace}\cite{FHT} 
A {\it semi-free filtration} of a DG module
$M$ over a DG algebra $A$ is a family $\{F^n\}_{n\in {\mathbb Z}}$ of
DG submodules of $M$ satisfying the condition:
$F^{-1}=0$, $F^n \subset F^{n+1}$, $\cup_{n\geq 0}F^n =M$ and 
$F^n/F^{n-1}$ is isomorphic to a direct sum of shifts of $A$. 
A module $M$ admitting a semi-free filtration is called {\it semi-free}.  
We say that the filtration $\{F^n\}_{n\in {\mathbb Z}}$ has 
{\it class at  most} $l$
if $F^l=M$ for some integer $l$. 
Moreover  $\{F^n\}_{n\in {\mathbb Z}}$ is called {\it finite} if 
the subquotients are finitely generated.   
\end{defn}

The above argument yields that $F^n$ is constructed from $F^{n-1}$ via the mapping cone construction.

\begin{thm}\cite[Theorem 4.2]{ABIM}
\label{thm:ABIM-I}
Let $M$ be a DG module over a DG algebra $A$ and $l$ a non-negative
integer. Then 
$\text{\em level}_{\text{\em D}(A)}(M)\leq l$ if and only if 
$M$ is a retract in $\text{\em D}(A)$ of some DG module admitting a finite
semi-free filtration of class at most $l-1$. 
\end{thm}

\section{Proofs of Theorem \ref{thm:main_new} and 
Proposition \ref{prop:main2}}

We may write $C^*(X)$ and $H^*(X)$ in place of $C^*(X; \K)$ and $H^*(X; \K)$,
respectively. 

\medskip
\noindent
{\it Proof of Theorem \ref{thm:main_new}.} We first prove the
assertion under the condition (i). 

Let ${\mathcal K} \to H^*(E) \to 0$ be the resolution of $H^*(E)$ which
is obtained by the two-sided Koszul resolution of 
$H^*(B)\cong \K[u_1, ..., u_m]$; that is, 
$$
({\mathcal K}, d) = (H^*(E)\otimes E[su_1, ..., su_m]\otimes 
\K[u_1, ..., u_m], d), 
$$ 
where $d(su_i)= q^*u_i\otimes 1 -1 \otimes u_i$ for $i = 1, ..., m$, 
$\text{bideg} \ y\otimes 1 = (0, \deg y)$ for $y \in H^*(E)$, 
$\text{bideg} \ su_i = (-1, \deg u_i)$, 
$\text{bideg} \ 1\otimes \ u_i = (0,  \deg u_i)$ and 
$E[su_1, ..., su_m]$ denotes the exterior algebra generated by 
$su_1, ..., su_m$; see \cite{B-S}.  
Thus in $\D(H^*(X))$, 
\begin{eqnarray*}
{\mathcal L}:= H^*(E)\otimes_{H^*(B)}^{{\mathbb L}}H^*(X)  
\! &\cong& \! 
({\mathcal K}\otimes_{H^*(B)}H^*(X), \delta) \\
&\cong& 
(H^*(E)\otimes E[su_1, ..., su_m]\otimes H^*(X), \delta),
\end{eqnarray*}
where $\delta(su_i)=q^*u_i\otimes 1 - 1\otimes \varphi^*u_i$. 
Put $s=\dim \Gamma$. Without loss of generality, it can be assumed that 
the set $\{u_{m-s+1}, ..., u_m\}$ is a 
basis of $\Gamma$.  
Let $\K\langle M \rangle$ denote the vector space spanned by a set
$M$, where $\K\langle \phi \rangle = \K$. 
Put $F^0=H^*(E)\otimes E[u_{m-s+1}, ..., u_m]\otimes H^*(X)$. 
For any integer $l$ with $1\leq l \leq m-s$, we define 
a DG submodule $F^l$ of ${\mathcal L}$ by 
\begin{eqnarray*}
F^l= H^*(E)\otimes E[su_{m-s+1}, ..., su_m] & &\\ 
&&\hspace{-4cm}\otimes  
\K\langle su_{i_1}\cdots su_{i_k} \ | \ 0\leq k \leq l, 
\ 1\leq i_1 < \cdots < i_k \leq m-s \rangle 
\otimes H^*(X). 
\end{eqnarray*}
Then $\{F^l\}_{0\leq l \leq m-s}$ is a finite semi-free filtration of 
${\mathcal L}$.
In fact, $\bigcup_{0\leq l \leq m-s}F^l={\mathcal L}$
and the quotient $F^l/F^{l-1}$ is isomorphic to 
a finite direct sum of shifts of $H^*(X)$ in $\D(H^*(X))$. 
Observe that 
$\text{Tor}^{\Lambda}(H^*(E), \K)=H(H^*(E)\otimes E[su_{m-s+1}, ...,
su_m])$ is of finite dimension by assumption. 
It follows from Theorem \ref{thm:ABIM-I}
that $\text{level}_{\text{D}(H^*(X))}({\mathcal L})$ is less
than or equal to $m-s+1$. 
In view of Theorem \ref{thm:level_formal}, we have 
$\text{level}_{\text{D}(C^*(X))}(E_{\varphi} ) =  
\text{level}_{\text{D}(H^*(X))}({\mathcal L}).  
$
One obtains the inequality. 

Suppose that the condition (ii) holds. 
It is immediate that $\text{Ker} \ (\Delta^* |_{QH^*(B'\times
B')})\cong \K\langle z_1\otimes 1 - 1\otimes z_1, ..., 
z_m\otimes 1 - 1\otimes z_m \rangle$. 
We have a free resolution of 
$H^*(B')$ as a right $H^*(B'\times B')$-module of the form 
$$
(E[sz_1, ..., sz_m]\otimes H^*(B'\times B'), \partial) \to H^*(B') \to 0 
$$
in which $\partial(sz_i)=z_i\otimes 1 -1\otimes z_i$ for $i = 1, ..., m$
; see 
\cite{S} \cite[Proposition 1.1]{K}.  
This enables us to conclude that in $\D(H^*(X))$
$$
H^*(B_\psi^I)\otimes^{\mathbb L}_{H^*(B\times B)}H^*(X)
\cong (E[sz_1, ..., sz_m]\otimes H^*(X), \widetilde{\partial})
$$
for which $\widetilde{\partial}(sz_i)=\psi^*(z_i\otimes 1 -1\otimes
z_i)$ for $i = 1, ..., m$.  
By adapting the above argument, we obtain the result.  
\hfill \qed

\begin{cor}
\label{cor:fibration} Let $F \to E \stackrel{q}{\to} B$ 
be a fibration for which $q$ is  $\K$-formalizable.  
Suppose that $H^*(B; \K)$ is a polynomial 
algebra generated by $m$ indecomposable elements. Then  
$
\text{\em level}_{\text{\em D}(C^*(E; \K))}(F)
\leq m-\dim (\text{\em Ker} \ q^* \cap QH^*(B; \K)) +1.
$ 
\end{cor}

\begin{proof}
By assumption $q$ is $\K$-formalizable; that is, $(q, \iota)$ is a
relatively $\K$-formalizable pair for some constant map 
$\iota : * \to B$. We choose the element $\iota(*)$ as a basepoint of
$B$. Consider the fibre square of the form 
$$
\xymatrix@C25pt@R12pt{
F_q \ar[r] \ar[d] & PB \ar[d]^{\pi} \\
E  \ar[r]^{q}  & B   
}
$$
in which $\pi : PB \to B$ is the path fibration. Observe that
$F_q$ is the homotopy fibre of $q$ and hence $F\simeq F_q$. 
Let $\iota' : * \to PB$ be a homotopy equivalence map with  
$\iota = \pi \iota'$. Then we have a diagram
$$
\xymatrix@C25pt@R12pt{
TV_B \ar[r]^{m_B}_{\simeq} \ar[d]^{\widetilde{\pi}} & C^*(B)
\ar[d]^{\pi^*} \ar@/^3pc/[dd]^{\iota^*}
\\
\K  \ar[r]_{m_{PB}}^{\simeq} \ar[dr]_{=} & 
C^*(PB) \ar[d]^{(\iota')^*}_{\simeq} \\   
& \ \ \ \ \ C^*(*)=\K,}
$$ 
where $\widetilde{\pi}$ and $m_{PB}$ are the augmentation and the unit,
respectively. Observe that $m_{PB}$ is a quasi-isomorphism. 
Since the outer square and the triangle are commutative
and $\iota^*=(\iota')^*\pi^*$, it follows from \cite[Theorem 3.7]{FHT}
that $\pi^*m_B\simeq m_{PB}\widetilde{\pi}$. This implies that 
$(q, \pi)$ is relatively $\K$-formalizable.  
It is immediate that the dimension of  
$\text{Tor}^{\K[\text{Ker}\ q^* \cap QH^*(B)]}
(H^*(PB), \K)$ is finite because $H^*(PB)=\K$. 
Theorem \ref{thm:main_new} yields the result. 
\end{proof}

We have a variant of Theorem \ref{thm:main_new}. 

\begin{prop}
\label{prop:homotopy_fibre}
Let ${\mathcal F}$ be the fibre square as in Theorem \ref{thm:main_new} for
which the condition {\em (i)} holds. 
Let $F_\varphi$ denote the homotopy fibre of $\varphi : X \to B$. 

{\em (1)} Suppose that $\text{\em Tor}^{\Lambda}_*(H^*(E; \K), \K)$ is a
trivial $H^*(B; \K)/(\Lambda^+)$-module. 
Then 
$$
\text{\em level}_{\text{\em D}(C^*(X; \K))}(E_{\varphi} )
= \text{\em level}_{\text{\em D}(C^*(X; \K))}(F_{\varphi}).  
$$

{\em (2)} Suppose that the cohomology $H^*(X; \K)$ is a polynomial algebra and 
the dimension of $H^*(F_\varphi)$ is finite. Then 
$$
\text{\em level}_{\text{\em D}(C^*(X; \K))}(F_{\varphi})
= \dim QH^*(X; \K)+1.$$
\end{prop}

\begin{proof}
With the same notation as in the proof of Theorem \ref{thm:main_new},  
we see that in $\D=\D(H^*(X))$ 
$$
{\mathcal L}  \cong (\text{Tor}^{\Lambda}(H^*(E), \K)\otimes 
E[su_1, ..., su_{m-s}]\otimes H^*(X), \delta). 
$$ 
Since $\text{Tor}^{\Lambda}(H^*(E), \K)$ is a trivial 
$H^*(B)/(\Lambda^+)$-module by assumption, it follows that 
$\delta(su_i)= -1\otimes \varphi^*(u_i)$ for $i = 1, ..., m-s$. Thus 
we see that 
$$
{\mathcal L}   \cong 
\bigoplus_i\Sigma^{l_i}\K\otimes_{\K[u_1, ..., u_{m-s}]}^{{\mathbb L}}H^*(X)  
$$ in $\D$  for some integers $l_i$. 
On the other hand,   
\begin{eqnarray*}
\K\otimes_{H^*(B)}^{{\mathbb L}}H^*(X)  \! & \cong & \! 
E[su_{m-s+1}, ..., su_m]\otimes E[su_1, ..., su_{m-s}]\otimes H^*(X) \\
& \cong & \bigoplus_k\Sigma^{l_k} E[su_1, ..., su_{m-s}]\otimes H^*(X)
\\
& \cong &  \bigoplus_k\Sigma^{l_k} 
\K\otimes_{\K[u_1, ..., u_{m-s}]}^{{\mathbb L}}H^*(X) 
\end{eqnarray*}
for some integers $l_k$. The result \cite[Lemma 2.4 (1)(3)]{ABIM} allows
us to conclude that 
\begin{eqnarray*}
\text{level}_{\D}(H^*(E)\otimes^{{\mathbb L}}_{H^*(B)}H^*(X)) \!& = & \!
\textstyle\mathop{\text{max}}\limits_i\{\text{level}_{\D}(\Sigma^{l_i}
\K\otimes_{\K[u_1, ..., u_{m-s}]}^{{\mathbb L}}H^*(X)) \}  \\
& = & \text{level}_{\D}(\K\otimes_{\K[u_1, ..., u_{m-s}]}^{{\mathbb L}}H^*(X)) \\
& = & \text{level}_{\D}
(\K\otimes_{H^*(B)}^{{\mathbb L}}H^*(X)) \\
& = & \text{level}_{\D(C^*(X))}(F_{\varphi}). 
\end{eqnarray*}  
The last equality follows from Theorem \ref{thm:level_formal} 
since $\varphi$ is $\K$-formalizable. 

Applying the result \cite[Corollary 5.7]{ABIM} to the DG module 
$\K\otimes_{H^*(B)}^{{\mathbb L}}H^*(X)$ over $H^*(X)$, we have the
latter half of the proposition. 
\end{proof}

\begin{rem}
\label{rem:ABIM}
Let ${\mathcal F}$ be the fibre square as in Theorem
\ref{thm:main_new}. Theorem \ref{thm:level_formal} and
\cite[Proposition 3.4 (1)]{ABIM} imply that 
$$
\text{level}_{\text{D}(C^*(X))}(E_{\varphi} ) =  
\text{level}_{\text{D}(H^*(X))}
(H^*(E)\otimes_{H^*(B)}^{{\mathbb L}}H^*(X) ) \leq 
\text{level}_{\text{D}(H^*(B))}(H^*(E)).  
$$
\end{rem}

\medskip
\noindent
{\it Proof of Proposition \ref{prop:main2}.}
Let $\{E_r, d_r\}$ and $\{\widehat{E}_r, \widehat{d}_r\}$ be 
the Eilenberg-Moore spectral sequence and 
the Leray-Serre spectral sequence for 
${\mathcal F}'$ with coefficients in $\K$, respectively.  
Since $\text{level}_{\D(C^*(B))}(E)=1$, it
follows from Theorem \ref{thm:ABIM-I}  
that $C^*(E)$ is a retract of a free $C^*(B)$-module 
of finite rank in $\D(C^*(B))$. 
Thus $H^*(E)$ is a projective $H^*(B)$-module and hence 
$$
\text{Tor}^{H^*(B)}_{-l, *}(H^*(E), \K) = 0 
\ \ \text{for} \ \ l >0. 
\eqnlabel{add-1}
$$ 
Since $E_2^{*,*}\cong \text{Tor}^{H^*(B)}_{-l, *}(H^*(E), \K)$, it
follows that $\{E_r, d_r\}$ collapses at the $E_2$-term. 

The induced map $j^* : H^*(E) \to H^*(F)$ factors through the edge
homomorphism $edge$ and coincides with the composite 
$$
\xymatrix@R12pt{
H^*(E) \ar@{->>}[r] & \text{Tor}^{H^*(B)}_{0, *}(H^*(E), \K) 
\ar[r]^(0.65){edge} & H^*(F). 
}
$$ 
By virtue of (3.1), we see that the edge homomorphism is an
isomorphism. This yields that $j^*$ is an epimorphism. 
Hence  $\{\widehat{E}_r, \widehat{d}_r\}$ collapses at the $E_2$-term. 
\hfill \qed

\medskip
\noindent
Proposition \ref{prop:main2} and the following lemma enable us 
to show that $E_\varphi$ in Theorem \ref{thm:main_new} is not of level 
one in some cases; see Section 5 for such examples. 

\begin{lem}
\label{lem:pullback}
Let ${\mathcal F}$ be the fibre square as in Theorem \ref{thm:main_new}. 
If the differential graded  module 
$H^*(E)\otimes_{H^*(B)}^{{\mathbb L}}H^*(X)$ 
is of level one, then so is $H^*(E_{\varphi})$. 
\end{lem}

\begin{proof}
The DG module $H^*(E)\otimes_{H^*(B)}^{{\mathbb L}}H^*(X)$ is a retract
of a free $H^*(X)$-module $\oplus \Sigma^{l_i}H^*(X)$. 
Thus $H^*(H^*(E)\otimes_{H^*(B)}^{{\mathbb L}}H^*(X))$
is a retract of $H^*(\oplus \Sigma^{l_i}H^*(X))=\oplus
\Sigma^{l_i}H^*(X)$. We see that 
$H^*(H^*(E)\otimes_{H^*(B)}^{{\mathbb L}}H^*(X))$ is in  
$\text{thick}^1_{\D(H^*(X))}(H^*(X))$. 
Since $(q, \varphi)$ is relatively $\K$-formalizable, it follows from 
\cite[Proposition 3.2]{K1} that $H^*(E_{\varphi})$ is isomorphic to 
$H^*(H^*(E)\otimes_{H^*(B)}^{{\mathbb L}}H^*(X))$ as an
$H^*(X)$-module. We have the result.
\end{proof}

\section{Examples}
By applying Theorem \ref{thm:main_new}, Proposition
\ref{prop:homotopy_fibre}, Proposition \ref{prop:main2} 
and some results in \cite{ABIM}, we obtain computational examples of the
cochain type levels of spaces. 

We begin by recalling an important space which appears in toric topology. 
Let $T^m$ be the $m$-torus and 
$D^2$ the disc in ${\mathbb C}$, namely   
$D^2 =\{ z \in  {\mathbb C} \ | \ |z| \leq 1 \}$. 
Let $V$ be the set of ordinals $[m]=\{ 1, 2, ..., m\}$. 
For a subset $w \subset V$, we define 
$$
B_w :=\{ (z_1, ..., z_m) \in (D^2)^m \ | \ |z_i| = 1 \ \text{for} \ 
i \notin w \}. 
$$
Let $S$ be an abstract simplicial complex with the vertex set $V$ and 
${\mathcal Z}_S$ denote the subspace $(D^2)^m$
defined by 
$$
{\mathcal Z}_S = \cup_{\sigma \in S}B_\sigma. 
$$
Then the $m$-torus $T^m$ acts on ${\mathcal Z}_S$ via 
the natural action of $T^m$ on $(D^2)^m$. 
We then have a Borel fibration ${\mathcal F}_S$ of the form 
${\mathcal Z}_S \to ET^m\times_{T^m} {\mathcal Z}_S \stackrel{q}{\to}
BT^m$. 
Let $DJ(S)$ be the Davis-Januszkiewicz space associated with the given
abstract simplicial complex $S$; that is, 
$$
DJ(S) = \cup_{\sigma \in S}(BT)_\sigma
$$ 
for which $(BT)_\sigma$ is the subspace of $BT^m$ defined by 
$$
(BT)_\sigma = \{(x_1, ..., x_m) \in BT^m \ | \ x_i=* \ \text{for} \ 
i \notin \sigma \}. 
$$
The Stanley-Reisner
algebra $\K[S]$ is defined to be the quotient graded algebra of the form 
$$
\K[t_1, ..., t_m]/(t_{i_1}\cdots t_{i_l} ; (i_1, ..., i_l) \notin S),
$$
where $\text{deg} \ t_i =2$ for any $i = 1, ..., m$.  
Observe that $H^*(DJ(S))$ is isomorphic to the Stanley-Reisner
algebra $\K[S]$ 
and $H^*({\mathcal Z}_S; \K)\cong \text{Tor}^{H^*(BT^m)}_*(\K[S], \K)$ 
as an algebra; see  \cite{B-P} and \cite{Pa}.

Since the construction of the Davis-Januszkiewicz space is natural
with respect to simplicial maps; that is, for a simplicial map 
$\phi : K \to S$, we have a map $DJ(K) \to DJ(S)$. In particular, the
inclusion of abstract simplicial complex $S$ with the vertex set $[m]$
to the standard $m$-dimensional simplicial complex $\Delta^{[m]}$ gives
rise to the inclusion $i : DJ(S) \to DJ(\Delta^{[m]})=BT^m$. 

The result \cite[Theorem 6.29]{B-P} due to Buchstaber and Panov asserts
that there exists a deformation retract 
$j : ET^m\times_{T^m} {\mathcal Z}_S \to DJ(S)$ such that 
the diagram 
$$
\xymatrix@C25pt@R15pt{
ET^m\times_{T^m} {\mathcal Z}_S \ar[r]^(0.6)p \ar[d]_j & BT^m \ar@{=}[d] \\
DJ(S) \ar[r]_i & BT^m
}
$$
is commutative. Thus we see that the homotopy fibre of the inclusion
$i : DJ(S) \to BT^m$ has the homotopy type of the moment-angle complex 
${\mathcal Z}_S$. The singular cochain complex $C^*(DJ(S))$ is viewed as a 
$C^*(BT^m)$-module via the induced map $C^*(i)$. We then have 

\begin{prop}
\label{prop:DJ} 
$
\text{\em level}_{\text{\em D}(C^*(BT^m))}(DJ(S)) = 
\text{\em sup}\{ i | \text{\em Tor}^{H^*(BT^m)}_{-i,*}(\K[S], \K)\neq 0\}+1.
$
\end{prop}

This result is proved by using formality of the Davis-Januszkiewicz
space and the following proposition. 
The proof is postponed to Section 7. 

\begin{prop} {\em (}cf. \cite[Corollary 4.10]{ABIM}{\em ) } 
\label{prop:pd} Let $p : E \to B$ be a fibration. 
If $p$ is $\K$-formalizable, then 
\begin{eqnarray*}
\text{\em level}_{\text{\em D}(C^*(B; \K))}(E)
&\!\!=\!\!& \text{\em pd}_{H^*(B; \K)}(H^*(E; \K))+1 \\
&\!\!=\!\!& \text{\em sup}\{ i | \text{\em Tor}^{H^*(B;
\K)}_{-i,*}(H^*(E; \K), \K)\neq 0\}+1. 
\end{eqnarray*}
\end{prop}

\medskip
\noindent
{\it Proof of Proposition \ref{prop:DJ}.} 
The result \cite[Theorem 4.8]{N-R} due to Notbohm and Ray implies that 
the cochain algebra $C^*(DJ(S); \K)$ is connected to the cohomology 
$\h{DJ(K)}$ with natural quasi-isomorphisms. Therefore, by applying the
lifting Lemma, we have a homotopy commutative diagram 
$$
\xymatrix@C25pt@R15pt{
\h{DJ(S)}  & TV_{DJ(S)} \ar[l]_(0.4){\simeq} 
\ar[r]^(0.4){\simeq} & C^*(DJ(S); \K)  \\
\h{BT^m} \ar[u]^{H^*(i)} 
& TV_{BT^m}  \ar[l]^(0.4){\simeq} \ar[r]_(0.4){\simeq} 
\ar[u]_{} & C^*(BT^m; \K) \ar[u]_{C^*(q)} 
}
$$
in which horizontal arrows are quasi-isomorphisms. 
Thus it follows that the pair $(i, id_{BT^m})$ is relatively $\K$-formalizable. 
The result follows from Proposition \ref{prop:pd}. 
\hfill\qed

\begin{rem}
Let $\omega$ be a subset of $[m]$ and $S_\omega$ the full subcomplex of a
simplicial complex $S$ defined by 
$S_\omega=\{ \sigma \in S \ | \ \sigma \subseteq \omega \}$. 
Hochster's result asserts that 
$$
\text{Tor}^{\K[t_1, .., t_m]}_{-i}(\K[S], \K)\cong 
\bigoplus_{w\subseteq [m]} 
\widetilde{H}^{|\omega|-1-i}(S_\omega; \K), 
$$
where $\widetilde{H}^{-1}(\phi)=\K$; see \cite[Theorem 5.1]{Pa}. 
Thus we have
$$
\text{level}_{\D(C^*(BT^m))}(DJ(S)) = 
\text{sup}\{i \ | \ \oplus_{w\subseteq [m]} 
\widetilde{H}^{|\omega|-1-i}(S_\omega; \K)\neq 0 \} +1. 
$$
\end{rem}

\medskip
We next deal with the level of a path space which fits into 
the fibre square ${\mathcal F}_2$ mentioned in the introduction.

Let $k$ be an integer and $\varphi_k : S^4 \to BSU(n)$ 
a representative of the element $k$ in $\pi_4(BSU(n))\cong {\mathbb Z}$.    
Let $\Delta : S^4 \to S^4\times S^4$ be the diagonal 
map. We have a fibre square of the form 
$$
\xymatrix@C25pt@R15pt{
B_{k}^I \ar[r] \ar[d] & BSU(n)^I 
\ar[d]^{\e_0\times \e_1} \\
S^4  \ar[r]_(0.3){(1\times\varphi_k)\Delta}  & BSU(n)\times BSU(n). 
}
$$

\begin{prop} Let $\K$ be a field of characteristic $p$. Then one has 
$$
\text{\em level}_{\text{\em D}(C^*(S^4))}(B_{k}^I)=
\left\{
\begin{array}{l}
2  \ \ \text{if} \ \  1-k \ \text{is not divisible by} \ p \\
1  \ \ \text{otherwise}.  
\end{array}
\right.
$$  
\end{prop}

\begin{proof}
We first observe that $H^*(BSU(n))$ is a polynomial algebra 
generated by $n-1$ indecomposable elements with even degree. 
It is immediate that 
$$\text{Ker} \ (\Delta^* |_{QH^*(B\times B)})
\cap \text{Ker}\ \psi_k^* = n-1-1
$$ 
if $1-k$ is not divisible by $p$,  
where $\psi_k=(1\times\varphi_k)\Delta$ and $B=BSU(n)$.  
By virtue of Theorem \ref{thm:main_new}, we have 
$
L:= \text{level}_{\text{D}(C^*(S^4))}(B_{k}^I) \leq 2.
$     
Suppose that $L=1$. 
By Lemma \ref{lem:pullback}, we see that 
$\text{level}_{\D(H^*(S^4))}(H^*((B_{k}^I))=1$. 
Proposition \ref{prop:main2} implies that 
$H^*(B_{k}^I)\cong H^*(\Omega BSU(n))\otimes H^*(S^4)$ as a vector space. 
On the other hand $H^*(B_{k}^I)=
H(H^*(BSU(n))\otimes^{\mathbb L}_{H^*(BSU(n)\times BSU(n))}H^*(S^d))\cong 
H(E[sz_1, ..., sz_{n-1}]\otimes H^*(X), \partial)$. 
Since $\psi_k^*\neq 0$, we have $\partial \neq 0$. This yields that 
$$\dim H^*(B_{k}^I)=\dim H(E[sz_1, ..., sz_{n-1}]
\otimes H^*(S^4), \partial)< \dim H^*(\Omega BSU(n)))\otimes H^*(S^4),$$ 
which is a contradiction. 
Hence $L=2$. 

If $1-k$ is divisible by $p$, then  
$\text{Ker} \ (\Delta^* |_{QH^*(B\times B)})
\cap \text{Ker}\ \psi_k^* = n-1$.  
Theorem \ref{thm:main_new} yields that $L=1$. 
\end{proof}

Let $G$ be a connected Lie group and $H$ a closed  
subgroup of $G$. Then we have a fibration of the form 
$
G/H \stackrel{i}{\to} BH \stackrel{Bj}{\to} BG.  
$
The induced map $i^* : C^*(BH; \K) \to C^*(G/H; \K)$ makes 
$C^*(G/H; \K)$ a DG module over $C^*(BH; \K)$. 

\begin{prop}
Suppose that $H^*(G; \K)$ and $H^*(H; \K)$ are
polynomial algebras with generators of even dimensions. Then 
$$
\text{\em level}_{\text{\em D}(C^*(BH; \K))}(G/H)=\dim
QH^*(BH; \K) +1. 
$$ 
\end{prop}

\begin{proof}
The induced map $Bj$ is $\K$-formalizable; see \cite[Section 7]{Mu}. 
Thus Proposition \ref{prop:homotopy_fibre} yields the result. 
\end{proof}

To prove Theorem \ref{thm:Borel}, we invoke the following lemma. 

\begin{lem}{\em (cf. }\cite[Proposition 3.4(1)]{ABIM}{\em )}
\label{lem:key_Borel}
Let $\psi : A \to B$ be a morphism of DG algebras and $M$  a DG module
over $B$. For a DG module $N$ over $B$, 
let $\psi_*N$ denote the DG module over $A$ via
$\psi$. Suppose that $\psi_*B$ is a finite direct sum of shifts of
$A$. Then 
$$
\text{\em level}_{\text{\em D}(A)}(\psi_*M)\leq 
\text{\em level}_{\text{\em D}(B)}(M). 
$$  
\end{lem}

\begin{proof}
Put $l=\text{level}_{\text{D}(B)}(M)$. 
It follows from Theorem \ref{thm:ABIM_I} that $M$ admits a finite semi-free
filtration $\{F^n\}_{0\leq n \leq l-1}$ of class at most $l-1$. By
definition, $F^n/F^{n-1}$ is isomorphic to a finite direct sum of
shifts of $B$. Therefore $\psi_*F^n/\psi_*F^{n-1}$ is isomorphic to a
direct sum of shifts of $A$ since so is $\psi_*B$. This completes the
proof. 
\end{proof}

\medskip
\noindent
{\it Proof of Theorem \ref{thm:Borel}.}
We first observe that $EG\times_HG/K$ and $EG\times_{T_H}G/T_K$ fit into
the fibre squares
$$
\xymatrix@C20pt@R2pt{
EG\times_HG/K \ar[r]^(0.6){\xi} \ar[dd] & BK \ar[dd]  &  &
             EG\times_{T_H}G/T_K  \ar[r]^(0.6){\xi} \ar[dd] & BT_K 
  \ar[dd] \\
& & \text{and} &   \\ 
BH  \ar[r]  & BG  &  &  BT_H  \ar[r]  & BG,   
}
$$
respectively, where $\xi$ sends $[x, g]$ to $[xg]$; 
see \cite[(2.2)]{K_fibration}. 
Thanks to Theorem \ref{thm:main_new},  
to prove the result, it suffices to show that 
$
\text{level}_{\text{D}(H^*(BH))}
(H^*(BK)\otimes_{H^*(BG)}^{\mathbb L}H^*(BH))$
is equal to 
$\text{level}_{\text{D}(H^*(BT_H))}
(H^*(BT_K)\otimes_{H^*(BG)}^{\mathbb L}H^*(BT_H))$. 
As is known \cite[6.3 Theorem]{Baum},  
$H^*(BT_H)\cong H^*(BH)\otimes H^*(H/T_H)$ as an $H^*(BH)$-module and 
$H^*(BT_K)\cong H^*(BK)\otimes H^*(H/T_K)$ as an $H^*(BK)$-module. 
Observe that these isomorphisms are also morphisms of $H^*(BG)$-modules. 
Thus one has
\begin{eqnarray*}
L_1 \! &:=& \! \text{level}_{\text{D}(H^*(BH))}
(H^*(BT_K)\otimes_{H^*(BG)}^{\mathbb L}H^*(BT_H)) \\ 
\! &=& \!  
\text{level}_{\text{D}(H^*(BH))}
(H^*(K/T_K)\otimes H^*(BK)\otimes_{H^*(BG)}^{\mathbb L}H^*(BT_H) )  \\
\! &=& \!\text{level}_{\text{D}(H^*(BH))}
(H^*(BK)\otimes_{H^*(BG)}^{\mathbb L}H^*(BT_H)) \\
\! &=& \!\text{level}_{\text{D}(H^*(BH))}
(H^*(BK)\otimes_{H^*(BG)}^{\mathbb L}H^*(BH)) \\ 
\! & \geq & \! 
\text{level}_{\text{D}(H^*(BT_H))}
(H^*(BK)\otimes_{H^*(BG)}^{\mathbb L}H^*(BT_H)) \\
\! &=& \!\text{level}_{\text{D}(H^*(BT_H))}
(H^*(BT_K)\otimes_{H^*(BG)}^{\mathbb L}H^*(BT_H)) =:L_2.
\end{eqnarray*}
The second and third equalities follow from 
\cite[Lemma 2.4 (1)(3)]{ABIM}. 
The existence of the exact functor $-\otimes^{\mathbb
L}_{H^*(BH)}H^*(BT_H) : D(H^*(BH)) \to \D(H^*(BT_H))$ yields the 
inequality; see \cite[Lemma 3.4 (1)]{ABIM}.   
In view of Lemma \ref{lem:key_Borel}, we have $L_1 \leq L_2$ and hence 
$L_1 = L_2$. This implies the result. 
\hfill \qed

\section{Proof of Theorem \ref{thm:cat-level} and computational examples}

We first mention that the proof of Theorem \ref{thm:cat-level} depends
heavily on the proof of \cite[Theorem 35.9]{F-H-T} and results due to
Kahl in \cite{Kahl}.   

We prove the first inequality. Let $A$ and $M$ denote the differential graded
algebra $C_*(\Omega Y)$ and the $C_*(\Omega Y)$-module $C_*(F_f)$,
respectively.  
Suppose that $\text{Ecat}_AM =n$. 
Then by definition, there exists a morphism 
$u : C_*(F_f) \to E_nA$ in the homotopy category of 
$\mathcal{DGM}\text{-}A$. Let $(R, d)$ be the Eilenberg-Moore resolution 
of $M$ \cite{F-H-T}\cite{G-M}. Then we have a composite
$$
\xymatrix@C20pt@R2pt{
(R, d) \ar[r]^(0.6){\simeq} & M \ar[rr]^(0.4){AW \circ C_*(\Delta)} & & M\otimes M
\ar[r]^(0.45){1\otimes u} & M\otimes E_nA 
}
$$
in the homotopy category of $\mathcal{DGM}\text{-}A$  
for which $A$ acts diagonally on the target, where 
$AW : C_*(X\times X) \to C_*(X)\otimes C_*(X)$ 
denotes the Alexander-Whitney map. 
Observe that the $A$-module $(R, d)$ has a semi-free filtration. 
Therefore by means of the lifting Lemma, we see that there exists a morphism 
$$
\xymatrix@C20pt@R2pt{
\psi : (R, d) \ar[r] & M \otimes E_nA = M \otimes 
T^{\leq n}(\Sigma\overline{A})\otimes A
}
$$
in $\mathcal{DGM}\text{-}A$. Thus we can proceed the proof of
\cite[Theorem 35.9]{F-H-T} from its Step 4 with the map $\psi$. 
In consequence,  we have the first inequality. 

Before proving the second inequality, we recall the definition of the trivial category in the sense of  Kahl; 
see \cite[Definition 2.1]{Kahl}. 

We call a morphism $f : P \to Q$ in $\mathcal{DGM}\text{-}A$ an elementary cofibration if there exists 
an inclusion $i : X \to Y$ between differential graded vector spaces such that $f $ is a cobase extension of the map 
$i\otimes id_A : X\otimes A  \to Y \otimes A$; that is, the morphism $f$ fits in the pushout diagram 
$$
\xymatrix@C15pt@R20pt{
X\otimes A \ar[d]_{i\otimes id_A} \ar[r] & P \ar[d]^{f} \\
Y\otimes A \ar[r]  &Q 
}
$$
in $\mathcal{DGM}\text{-}A$ for an appropriate morphism 
$X\otimes A \to P$ of supplemented DG modules over $A$. 
We denote an elementary cofibration by $\xymatrix@C10pt{\!\! f : 
P \ \ar@{>->}[r] & Q}$. 

\begin{defn} Let $M$ be an object of $\mathcal{DGM}\text{-}A$.
The {\it trivial category} of $M$, denoted $\text{trivcat}_A M$, is the least integer $n$ for which there exists a sequence 
$\xymatrix@C10pt@R2pt{
P^0 \ \ar@{>->}[r] & \cdots   \ \ar@{>->}[r] & P^n 
} $ of elementary cofibrations such that $P^0$ is a free $A$-module and $P^n$ is isomorphic to $M$ in the homotopy category 
of $\mathcal{DGM}\text{-}A$. If no such integer exists, we set $\text{trivcat}_A M=\infty$.
\end{defn}

\begin{lem} \label{lem:level-trivcat}
Let $M$ be an object in $\mathcal{DGM}\text{-}A$. Then there exists an
object $M'$ in $\mathcal{DGM}\text{-}A$ such that $M$ is a retract of
$M'$ in the homotopy category $\mathcal{DGM}\text{-}A$  and 
$$
\text{\em level}_{D(A)}M -1 \geq \text{\em trivcat}_AM'.  
$$
\end{lem}

\begin{proof}
Suppose that $\text{level}_{D(A)}M =l$. By virtue of Theorem
\ref{thm:ABIM-I}, 
there exists an DG-module $M'$ such that $M$ is a retract of $M'$ in
$\D(A)$ and $M'$ admits a finite semi-free resolution $\{F^n\}_{n\geq -1}$ 
of class at most $l-1$. Since $M$ is supplemented and 
$M'$ is connected to a DG-module of the
form $M\oplus N$ for a DG-module $N$ with quasi-isomorphisms, it
follows that $M'$ is also supplemented. 
We write $M'=\K\oplus \overline{M'}$ for which $d1 =0$ and 
$d(\overline{M'}) \subset \overline{M'}$. 

Suppose that there exists an integer $i$ such that $F^{i+1}$ is
supplemented but not $F^i$. Thus we see that 
$F^{i+1}\cong F^i\oplus \Sigma(Z^{i+1}\otimes A)$ as an 
$A^\natural$-module 
for which $Z^{i+1}$ is a finite dimensional graded vector space endowed
with the trivial differential.  
We may further assume that $\K$ is a direct summand of $Z^{i+1}_{-1}$
and the element $1 \in \K \subset Z^{i+1}$ corresponds to the element 
$1$ in $F^{i+1}$ under the isomorphism mentioned above. 

We shall construct a sequence of elementary cofibrations with $M'$ as
the target. Put $\widetilde{F}^s = F^s$ for $s > i$ and 
 $\widetilde{F}^s = F^s \oplus A$ for $s \leq i$. 
We define $\iota_s : \widetilde{F}^s \to \widetilde{F}^{s+1}$ by 
$\widetilde{\iota}_s = \iota_s$ for $s > i$ and  
$\widetilde{\iota}_s = \iota_s\oplus id$ for $s < i$, where 
$\iota_s : F^s \to F^{s+1}$ denotes the inclusion. 
Moreover define 
$
\widetilde{\iota}_i : 
\widetilde{F}^i =F^i\oplus A \to F^{i+1}\cong F^i\oplus \Sigma(Z^{i+1}\otimes A) 
$
by $\widetilde{\iota}_i(1)=\Sigma 1$ for $1\in A$ and 
$\widetilde{\iota}_i(w)=\iota(w)$ for $w \in F^i$. 
We write $Z^{i+1}=\K\oplus \overline{Z}$. 
Then it follows that 
$$
F^i\oplus \Sigma(Z^{i+1}\otimes A) \cong F^i\oplus (\Sigma \K \otimes A\oplus \Sigma (\overline{Z}\otimes A)) 
    \cong \widetilde{F}^i\oplus \Sigma(\overline{Z}\otimes A). 
$$
Consider the pushout diagram 
$$
\xymatrix@C15pt@R20pt{
(\overline{Z} \oplus \K)\otimes A \ar@{=}[r] & (\overline{Z} \otimes A)\oplus A \ar[r]^\xi  \ar[d]_i & F^i\oplus A = \widetilde{F}^i 
\ar[d] \\
(C\overline{Z} \oplus \K)\otimes A \ar@{=}[r] & C(\overline{Z} \otimes A)\oplus A \ar[r] & 
\widetilde{F}^i\cup_\xi(C(\overline{Z} \otimes A)\oplus A))
}
$$
in $\mathcal{DGM}\text{-}A$ in which $\xi$ is a morphism of DG
$A$-modules defined by 
$
\xi(\overline{z}\otimes a)= d(\Sigma (\overline{z} \otimes a))
-(-1)^{\deg \Sigma\overline{z}}\Sigma \overline{z}\otimes da
= d(\Sigma(\overline{z}\otimes 1))a
$ for 
$\overline{z} \otimes a \in \overline {Z}\otimes A$ and $\xi(a)= a$ for $a \in 0\oplus A$; see Section 1. 
We then see that 
\begin{eqnarray*}
\widetilde{F}^i\cup_\xi(C(\overline{Z} \otimes A)\oplus A))&\! = \!&
F^i\oplus A \oplus C(\overline{Z} \otimes A)\oplus A/(\xi(w)-i(w) ; w \in (\overline{Z} \otimes A)\oplus A)\\
&\cong& \{F^i \oplus C(\overline{Z} \otimes A)/(\xi(w)-i(w) ; w \in (\overline{Z} \otimes A)) \}\oplus A \\
&\cong& F^i \oplus A\oplus \Sigma(\overline{Z} \otimes A) \ \cong \ \widetilde{F}^{i+1}. 
\end{eqnarray*}
Thus the inclusion $\widetilde{F}^i \to \widetilde{F}^{i+1}$ is an
elementary cofibration. 

In the case where $s >i$, there exists  an inclusion $\iota_s : F^s \to F^{s+1}\cong F^s \oplus \Sigma (Z\otimes A)$ 
such that $\iota_s(1) = 1$.  Thus we obtain a pushout diagram 
$$
\xymatrix@C15pt@R20pt{
(Z \oplus \K)\otimes A \ar@{=}[r] & (Z \otimes A)\oplus A \ar[r]^{\xi'}  \ar[d]_i & F^s 
\ar[d] \\
(CZ \oplus \K)\otimes A \ar@{=}[r] & C(Z \otimes A)\oplus A \ar[r] & 
F^s\cup_{\xi'}(C(Z \otimes A)\oplus A))
}
$$
in $\mathcal{DGM}\text{-}A$ in which $\xi'$ is defined by $\xi'(z)= d(\Sigma (z \otimes 1))$ for 
$z \in Z$ and $\xi'(1)= 1$ for $1 \in 0\oplus A$. 
It follows that 
$$
F^s\cup_{\xi'}(C(Z \otimes A)\oplus A))\cong F^s\oplus C(Z\otimes A)/(\xi'(w) - i(w) ; w \in Z\otimes A) \cong F^{s+1}.
$$

In the case where $s < i$,  we obtain a pushout diagram 
$$
\xymatrix@C15pt@R20pt{
(Z \oplus \K)\otimes A \ar@{=}[r] & (Z \otimes A)\oplus A \ar[r]^{\zeta \oplus id_A}  \ar[d]_i & F^s\oplus A  = \widetilde{F}^s
\ar[d] \\
(CZ \oplus \K)\otimes A \ar@{=}[r] & C(Z \otimes A)\oplus A \ar[r] & 
(F^s \oplus A)\cup_{\zeta\oplus id_A}(C(Z \otimes A)\oplus A))
}
$$
in $\mathcal{DGM}\text{-}A$ in which $\zeta$ is defined by $\zeta(z)= d(\Sigma (z \otimes 1))$ for 
$z \in Z$. 
It follows that 
$$
(F^s \oplus A)\cup_{\zeta\oplus id_A}(C(Z \otimes A)\oplus A))
\cong \{F^s \cup_{\zeta}C(Z \otimes A)\}\oplus A 
\cong F^{s+1}\oplus A = \widetilde{F}^{s+1}.
$$
The above argument enables us to obtain a sequence of elementary cofibrations 
$$ 
\xymatrix@C12pt@R2pt{
\widetilde{F}^{0} \  \ar@{>->}[r] & \widetilde{F}^{2} \ \ar@{>->}[r] &
\cdots \ \ar@{>->}[r] & \widetilde{F}^{i} \ \ar@{>->}[r] & \widetilde{F}^{i+1} 
\ \ar@{>->}[r] & \cdots \ \ar@{>->}[r] & \widetilde{F}^{l-1}=M'. 
}
$$ 
This yields that $\text{trivcat}_AM' \leq l-1$.
\end{proof}

We are now ready to prove the second inequality. Let $M'$ be the supplemented DG-module described in 
Lemma \ref{lem:level-trivcat}. The result \cite[Theorem 2.6]{Kahl} allows us to conclude that 
$n:=\text{trivcat}_A M' \geq \text{Ecat}_AM'$. Thus there exists a morphism $M' \to E_nM$ in the homotopy category 
$\text{Ho}(\mathcal{DGM}\text{-}A)$ . Since $M$ is a retract of $M'$, we have a morphism $M \to M'$ in  
$\text{Ho}(\mathcal{DGM}\text{-}A)$. This implies that  
$\text{Ecat}_AM' \geq \text{Ecat}_AM$. We have the second inequality. 

The result \cite[Lemma 6.5]{Schmidt} implies that 
$
\text{level}_{\D(A)}M \leq \dim H(M\otimes_A^{\mathbb L} \K). 
$
In our case, we have $H(M\otimes_A^{\mathbb L} \K)=H(C_*(F_f)\otimes _{C_*(\Omega Y)}^{\mathbb L} \K) \cong H_*(X; \K)$. 
The isomorphism follows from the Eilenberg-Moore theorem; see for example \cite[Theorem 3.9]{G-M}. 
This enables us to obtain the last inequality.  

The latter half of the assertion follows from Lemma \ref{lem:app1}
below.  

\medskip
\noindent
{\it Proof of Corollary \ref{cor:R-case}}. 
As described in Section 2, the E-category in $\mathcal{DGM}-C_*(\Omega X)$ coincides with the M-category; that is, 
we have 
$\text{Ecat}_{C_*(\Omega X)}\Q=\text{Mcat} (TV, d)$, where 
the right hand side denotes the M-category of a TV-model for $X$ in the sense of Halperin and Lemaire \cite{H-L}.
It follows from \cite[Theorem 3.3 (ii)]{H-L} and the main theorem in 
\cite{Hess} that  
$\text{Mcat} (TV, d)=\text{cat} X$. We have the result. 
\hfill \qed

\begin{ex}
\label{ex:5.1}
Let $X$ be a simply-connected space whose cohomology with coefficients in $\K$ is generated by a single element $x$. 
Suppose that $x^l\neq 0$ and $x^{l+1}=0$.  
We compute $\text{level}_{\D(C_*(\Omega X ; \K))}\K$.  
Theorem \ref{thm:cat-level} yields that 
$$
\text{Mcat}(TV) = 
\text{Ecat}_{C_*(\Omega X ; \K)}\K \leq \text{level}_{\D(C_*(\Omega X ; \K))}\K-1 
\leq \dim H^*(X; \K)-1, 
$$  
where $(TV, d)$ is a TV-model for $X$. 
The result \cite[Proposition 1.5]{H-L} implies that 
the cup length $c(X)$ of $H^*(X; \K)$ is a lower bound of the M-category.
Thus we have $\text{level}_{C_*(\Omega X ; \K)}\K=l+1$. 
\end{ex}

\begin{ex}
\label{ex:5.2}
We next consider $\text{level}_{C_*(\Omega X ; \Q)}\Q$ for 
a simply-connected rational H-space $X$ with 
$\dim H^*(X; Q)<\infty$. 
It follows that $H^*(X; \Q)$ is isomorphic, as a Hopf algebra,  
to the exterior algebra generated by primitive elements with odd degrees, say 
$H^*(X; \Q)=\wedge (x_1, ..., x_l)$. We see that 
$H_*(\Omega X; \Q)\cong \Q[y_1, ..., y_l]$ as an algebra, where 
$\deg y_i = \deg x_i-1$. Theorem \ref{thm:cat-level} and 
Corollary \ref{cor:R-case} yield that 
$$
l=c(X) \leq \text{cat} X =\text{level}_{\D(C_*(\Omega X ; \Q))}\Q -1 
\leq \text{pd}_{H_*(\Omega X)}\Q = l.
$$
We have $\text{cat}  X + 1=\text{level}_{\D(C_*(\Omega Y ; \Q))}\Q = l+1$. 
\end{ex}

\section{Lower and upper bounds of the levels}

In this section, we prove Proposition \ref{prop:pd}. 
To this end, we need lemmas. 

Throughout this section, it is assumed that a DGA $A$ is non-negative;
that is, $A^i=0$ for $i < 0$.  

\begin{lem}{\em (}cf. \cite[Theorem 5.5]{ABIM}{\em )}
\label{lem:app1} Let $A$ be a non-negatively graded DGA over a field
$\K$ with $H^0(A)=\K$ and 
$M$ a DG module over A. Suppose that 
there exists an integer $N$ such that $H^j(M)=0$ for $j < N$. 
Assume further that 
$\text{\em Tor}_{-i}^{H(A)}(H(M), \K)$ is of finite dimension for any 
$i \leq 0$. Then one 
has 
$$\text{\em level}_{\text{\em D}(A)}(M) \leq 
\text{\em pd}_{H(A)}(H(M))+1 = 
\text{\em sup}\{ i | \text{\em Tor}_{-i}^{H(A)}(H(M), \K)\neq 0 \}+1.
$$ 
The same assertion as above holds for the homological case. 
\end{lem}

\begin{proof} Suppose that 
$\text{pd}_{H(A)}(H(M))+1=l < \infty$. 
Then we have a projective resolution of $H(M)$ as a right $H(A)$-module of
the form 
$$
0 \to P_l \to \cdots \to P_1 \to P_0 \to H(M) \to 0. 
$$ 
We can assume that $P_{l,k}=0$ for $k < N$. 
Since $H^0(A)=\K$, it follows from 
\cite[12.2.8 Theorem]{M-R} that each $P_i$ is a free $H(A)$-module, say 
$P_i=X_i \otimes_\K H(A)$; see also \cite[page 274, Remark 1]{F-H-T}.
We further assume that the resolution is minimal. 
Observe that the same argument as in the proof of \cite[Theorem 2.4]{Roberts} is
applicable when constructing a minimal projective resolution of $H(M)$
as an $H(A)$-module since $H(M)$ and $H(A)$ are locally finite.  
By assumption, we have 
$\dim \text{Tor}_{-i}^{H(A)}(H(M), \K) < \infty$ for $i \leq 0$. 
This fact yields that $\dim X_i < \infty$. 

The result 
\cite[Theorem 2.1]{G-M}
implies that there exists a quasi-isomorphism  
$\oplus_i\widetilde{P}_i \stackrel{\simeq}{\to} M$ such that   
$\widetilde{P}_i = X_i\otimes_\K A$ as an $A^\natural$-module and that  
$\oplus_i\widetilde{P}_i$ admits semi-free filtration $\{F^n\}$ 
with $F^n = \oplus_{i\leq n}\widetilde{P}_i$.  
By Theorem \ref{thm:ABIM-I}, we have 
$\text{level}_{\text{D}(A)}(M) \leq l+1$.
This completes the proof. 
\end{proof}

\begin{lem}
\label{lem:app2}
Let $A$ be a graded algebra over a field and $M$ a right
$A$-module. Then 
$\text{\em pd}_{A}(M)+1 \leq \text{\em level}_{\text{\em D}(A)}(M)$.
\end{lem}

\begin{proof}
The assertion follows from the proof of the result \cite[Lemma
2.4]{Kr-Ku} due to Krause and Kussin.  
\end{proof}

By Lemmas \ref{lem:app1} and \ref{lem:app2}, we have 

\begin{cor}
Under the same assumption as in Lemmas \ref{lem:app1}, 
$$
\text{\em level}_{\text{\em D}(A)}(M) \leq 
\text{\em level}_{\text{\em D}(H(A))}(H(M)). 
$$
\end{cor}

\medskip
\noindent
{\it Proof of Proposition \ref{prop:pd}.}
Lemma \ref{lem:app1} and its proof yield that 
$$\text{level}_{\text{D}(C^*(B))}(E) \leq 
\text{pd}_{H^*(B)}(H^*(E))+1 =
\text{sup}\{ i | \text{Tor}_{-i}^{H^*(B)}(H^*(E), \K)\neq 0)+1. 
$$ 
We view the fibration $p : E \to B$ as a pull-back of itself by the
identity map $B\to B$. 
Since $p$ is $\K$-formalizable, it follows that $(p, id_B)$ is 
a $\K$-formalizable pair. By Theorem \ref{thm:level_formal}, we see that 
\begin{eqnarray*}
\text{level}_{\text{D}(C^*(B))}(E) 
\! &=& \!  
\text{level}_{\text{D}(H^*(B))}
(H^*(E)\otimes_{H^*(B)}^{{\mathbb L}}H^*(B) )  \\
\! &=& \!\text{level}_{\text{D}(H^*(B))}(H^*(E)).  
\end{eqnarray*} 
Lemma \ref{lem:app2} implies that 
$\text{pd}_{H^*(B; \K)}(H^*(E; \K))+1 \leq 
\text{level}_{\text{D}(H^*(B))}(H^*(E))$. We have the result.  
\hfill\qed

\begin{rem}
The inequality in Lemma \ref{lem:app1} may be strict. For example, we consider the Hopf map 
$S^3 \to S^2$ with fibre $S^1$. Then $C^*(S^3 ; \K)$ is viewed as
$C^*(S^2; \K)$-module via the Hopf map and hence it is 
in $\D(C^*(S^2; \K))$. The results \cite[Proposition 2.10]{K2} and
\cite[Proposition 6.6]{Schmidt} allow us to conclude that   
$\text{level}_{\D(C^*(S^2; \K))}(S^3)=2$. 
On the other hand, we can
construct a minimal projective resolution of 
$H^*(S^3; \K)$ as an $H^*(S^2; \K)$-module of the form 
$$
{\mathcal K}=(\wedge(x_3)\otimes \Gamma[w]\otimes \wedge (s^{-1}x_2) \otimes
\K[x_2]/(x_2^2), d) \to \wedge(x_3) \to 0, 
$$
for which $d(w)=s^{-1}x_2\cdot x_2$, $\text{bideg} \ s^{-1}x_2=(-1, 2)$, 
$\text{bideg} \ w=(-2, 4)$,  
$x_2$ and $x_3$ are generators of $H^*(S^2; \K)$ and $H^*(S^3; \K)$,
respectively. Here $\Gamma[w]$ denotes the divided power algebra
generated by $w$.  We see that 
$$
\text{Tor}^{H^*(S^2; \K)}_*(H^*(S^3; \K), \K)=
\wedge(x_3)\otimes \Gamma[w]\otimes \wedge (s^{-1}x_2). 
$$
It is readily seen that the torsion product is of infinite dimension.   
This implies that $\text{pd}_{H^*(S^2; \K)}(H^*(S^3; \K))=\infty$. 
\end{rem}

\medskip
We conclude this section by deducing a lower bound of the level.  

Let $A$ be a DGA. 
Following Hovey and Lockridge \cite{Ho-L}, we call a map $f: M \to N$ in
$\D(A)$ a ghost if $H(f)=0$. Moreover $M \in \D(A)$ is said to have
ghost length $n$, denoted $\text{gh.len.} M = n$, if every composition 
$$
M \stackrel{f_1}{\to} Y_1  \stackrel{f_2}{\to} \cdots
\stackrel{f_{n+1}}{\to} Y_{n+1}
$$
of $n+1$ ghost is trivial in $\D(A)$, and there exists a composite of
$n$ ghosts from $M$ is non trivial in $\D(A)$.

\begin{prop} \cite[Lemma 6.7]{Schmidt}
\label{prop:ghost}
For any $M \in \D(A)$, one has 
$$
\text{\em gh.len.} M +1 \leq \text{\em level}_{\D(A)}(M).  
$$
\end{prop}

In order to prove Proposition \ref{prop:ghost}, we recall the so-called Ghost lemma. 

\begin{lem}\cite{R} \cite[Lemma 2.3]{Kr-Ku}
\label{lem:Krause-Kussin}
Let $\D$ be a triangulated category and let 
$$
H_1 \stackrel{F_1}{\to} H_2 \stackrel{F_2}{\to}\cdots \stackrel{F_{n+1}}{\to}H_{n+1}  
$$
be a sequence of morphism between cohomological functors 
$\D^{\text{\em op}} \to \mathcal{A}b$. 
Let ${\mathcal X}$ be a subcategory of $\D$ such that $F_i$ vanishes on 
$\text{\tt thick}^1_{\D}({\mathcal X}) = 
\text{\tt smd}((\text{\tt add}^{\Sigma}({\mathcal X})))$. Then the
composite 
$F_n \circ \cdots \circ F_1$ vanishes on $\text{\tt thick}^n_{\D}({\mathcal X})$. 
\end{lem}

\noindent
{\it Proof of Proposition \ref{prop:ghost}}. For an object $M \in
\D(A)$, suppose that there exists a composite $
M \stackrel{f_1}{\to} Y_1  \stackrel{f_2}{\to} \cdots
\stackrel{f_{n}}{\to} Y_{n}$ 
of $n$ ghosts which is non trivial in $\D(A)$. 
We have a sequence
$$
\text{Hom}_{\D(A)}(- , M) \stackrel{(f_1)_*}{\to} 
\text{Hom}_{\D(A)}(- , Y_1)  \stackrel{(f_2)_*}{\to} \cdots
\stackrel{(f_{n})_*}{\to} \text{Hom}_{\D(A)}(- , Y_{n})
$$
of morphisms between cohomological functors. 
Since $\text{Hom}_{\D(A)}(\Sigma^{-n}A , M)=H^n(M)$ and each $f_i$ is a
ghost, it follows that 
$(f_i)_*$ is vanishes on $\text{\tt thick}^1_{\D}(A)$. 
By Lemma \ref{lem:Krause-Kussin}, we see that   
$(f_n)_* \circ \cdots \circ (f_1)_*$ vanishes on $\text{\tt thick}^n_{\D}(A)$. 
Thus if $M$ is in $\text{\tt thick}^n_{\D}(A)$, then 
$f_n\circ \cdots \circ f_1 = (f_n)_*\circ \cdots \circ (f_1)_*(id_M) :  M \to Y_n$
is trivial in $\D(A)$, which is a contradiction and hence 
we have the result. \hfill\qed

%




\medskip
\noindent
{\it Acknowledgments.} 
I am grateful to Norio Iwase for drawing my attention to the relationship
between the L.-S. category and the level. I also thank the referee for careful reading of a previous version of this paper.

\end{document}